\documentclass[times,sort&compress,3p]{elsarticle}
\journal{Journal of Multivariate Analysis}
\usepackage[labelfont=bf]{caption}

\usepackage{amsmath,amsfonts,amssymb,amsthm,booktabs,color,epsfig,graphicx,hyperref,url}
\usepackage{tikz}
\usepackage{comment}
\usepackage{relsize}
\theoremstyle{plain}

\newtheorem{theorem}{Theorem}[section]
\newtheorem{lemma}[theorem]{Lemma}
\newtheorem{prop}[theorem]{Proposition}
\newtheorem{cor}[theorem]{Corollary}

\theoremstyle{definition}
\newtheorem{defi}[theorem]{Definition}
\newtheorem{example}[theorem]{Example}

\theoremstyle{definition}
\newtheorem{remark}[theorem]{Remark}

\newcommand{\bthe}{\begin{theorem}}
	\newcommand{\ethe}{\end{theorem}}

\newcommand{\ble}{\begin{lemma}}
	\newcommand{\ele}{\end{lemma}}

\newcommand{\bde}{\begin{definition}\rm}
	\newcommand{\ede}{\halmos\end{definition}}

\newcommand{\bco}{\begin{corollary}}
	\newcommand{\eco}{\end{corollary}}

\newcommand{\bpr}{\begin{proposition}}
	\newcommand{\epr}{\end{proposition}}

\newcommand{\brem}{\begin{remark}\rm}
	\newcommand{\erem}{\end{remark}}

\newcommand{\bproof}{\begin{proof}}
	\newcommand{\eproof}{\end{proof}}

\newcommand{\bexam}{\begin{example}\rm}
	\newcommand{\eexam}{\end{example}}

\numberwithin{equation}{section}

\newcommand{\beao}{\begin{eqnarray*}}
	\newcommand{\eeao}{\end{eqnarray*}\noindent}

\newcommand{\beam}{\begin{eqnarray}}
\newcommand{\eeam}{\end{eqnarray}\noindent}

\newcommand{\barr}{\begin{array}}
	\newcommand{\earr}{\end{array}}

\newcommand{\nat}{{\mathbb N}}
\newcommand{\ganz}{{\mathbb Z}}

\newcommand{\an}{\operatorname{an}}
\newcommand{\An}{\operatorname{An}}
\newcommand{\De}{\operatorname{De}}

\newcommand{\pa}{\operatorname{pa}}

\def\N{{\mathbb N}}

\def\Z{{\mathbb Z}}

\newcommand{\nto}{n\to\infty}

\newcommand{\stas}{\stackrel{\rm a.s.}{\rightarrow}}

\newcommand{\al}{{\alpha}}

\newcommand{\ov}{\overline}

\allowdisplaybreaks

\begin{document}

\begin{frontmatter}

\title{Max-linear models in random environment}

\author[1]{Claudia Kl\"uppelberg}
\author[2]{Ercan S\"onmez\corref{mycorrespondingauthor}}

\address[1]{Center for Mathematical Sciences, Technical University of Munich, Boltzmannstr. 3, D--85748 Garching, Germany}
\address[2]{Department of Statistics, University of Klagenfurt, Universit\"atsstra{\ss}e 65--67, 9020 Klagenfurt, Austria}
\cortext[mycorrespondingauthor]{Corresponding author. Email address: \url{ercan.soenmez@aau.at}}

\begin{abstract}
We extend previous work of max-linear models on finite directed acyclic graphs to infinite graphs as well as random graphs, and investigate their relations to classical percolation theory, more particularly the impact of Bernoulli bond percolation on such models. We show that the critical probability of percolation on the oriented square lattice graph $\Z^2$ describes a phase transition in the obtained model. Focus is on the dependence introduced by this graph into the max-linear model. We discuss natural applications in communication networks, in particular, concerning the propagation of influences.
\end{abstract}

\begin{keyword} 
Bernoulli bond percolation\sep
extreme value theory \sep
graphical model\sep
infinite graph\sep 
percolation\sep
recursive max-linear model.
\MSC[2020] Primary 60G70 \sep
Secondary 60K35, 62--09.
\end{keyword}

\end{frontmatter}

\section{Introduction\label{s1}}

Extreme value theory is concerned with max-stable random elements which occur as limits of normalized maxima.
The theory has progressed in recent years from classical finite models to infinite-dimensional models (see, for example, \cite{EKM, Resnick1, Resnick2}).
A monograph relevant in the infinite-dimensional context is \cite{deHF}.
Prominent models are stochastic processes in space and/or time having finite dimensional max-stable distributions (e.g., see \cite{deH,GHV,KSH}).
Such processes model extreme dependence between process values at different locations and/or time points.

Max-linear models are natural analogues of linear models in an extreme value framework.
Within the class of multivariate extreme value distributions, whose dependence structures are characterized by a measure on the sphere, they are characterized by the fact that this measure is discrete  (e.g. \cite{ST}). 

In this paper we connect two research fields, namely max-linear models on directed acyclic graphs and percolation theory. Directed acyclic graphs, also called Bayesian networks, describe conditional independence properties between random variables. Percolation, in particular Bernoulli bond percolation is a simple way of obtaining a random version of a directed acyclic graph using a sample of iid Bernoulli random variables.

We extend previous work of max-linear models on finite directed acyclic graphs  (e.g. \cite{GK,GKO,KL}) to infinite graphs.
The model allows for finite subgraphs with different dependence structures, and
we envision applications where this may play a role as, for instance, a hierarchy of communities with different communication structures. Max-linear models on directed acyclic graphs have been the subject of concrete useful applications, for example in \cite{Ein} they have been fitted in order to explain properties of European stock markets, in which the economic sector influences the tail behavior of stock returns by means of max-linear behavior. The model we propose is quite flexible, as we work on arbitrary subgraphs of the oriented 2-dimensional lattice, additionally incorporate randomness. Thus our model allows to capture arbitrary (finite) directed acyclic graphs by identifying their edges with paths in our model. Therefore, such directed acyclic graphs which can be fitted in the description of European stock markets are included in our model as well.

We investigate the relation of the infinite max-linear model to classical percolation theory, more precisely to nearest neighbor bond percolation (e.g. \cite{BR,Grimmett}). 
We focus on the square lattice $\Z^2$ with edges to the nearest neighbors, where we orient all edges in a natural way (north-east) resulting in a directed acyclic graph (DAG) on this lattice.
On this infinite DAG a random sub-DAG may be constructed by choosing nodes and edges between them at random. In a Bernoulli bond percolation DAG edges are independently declared open with probability $p\in [0,1]$ and closed otherwise.  The random graph consists then of the nodes and the open edges. 
The {percolation probability} is the probability $P_p(|C(i)|=\infty)$, with $|C(i)|$ denoting the cardinality of $C(i)$, that a given node $i$ belongs to an infinite open cluster $C(i)$, which is 0 if $p\le 1/2$ and positive for $p>1/2$.
Kolmogorov's zero-one law entails that  an infinite open cluster exists for $p> 1/2$ with probability 1, and otherwise with probability 0.

We combine percolation theory with an infinite max-linear model by assigning to each node a max-linear random variable.
Sampling a random graph by Bernoulli bond percolation, we use this subgraph for modelling the dependence in the max-linear process on the oriented square lattice.
The max-linear models we envision are recursively constructed from independent continuously distributed random variables $(Z_j)_{j\in\Z^2}$, which include the class of variables belonging to the max-domain of attraction of the Fr\'echet distribution. 
More precisely, each random variable $X_i$ on a node $i\in\Z^2$ with ancestral set $\an(i)$ exhibits the property
\beam\label{fimodel}
X_i=\bigvee_{j\in\{i\}\cup\an(i)} b_{ij} Z_j,
\eeam
in distribution on every finite DAG, where $b_{ji}$ are positive coefficients. 
As this model is defined on a random graph it is a {max-linear model in random environment}. 
According to our terminology models investigated in \cite{Jelen2,Marko0} can also be seen as models in random environment. For related work we also refer to \cite{Lebed}. To the best of our knowledge, our model is the first such model studying the impact of Bernoulli bond percolation on max-linear models in the sense that we show that classical results regarding two different phases of Bernoulli bond percolation can be transferred into two distinct phases of typical behavior of naturally investigated properties of max-linear models, namely the dependence structure in max-linear models. 

One prerequisite for this work is the fact that max-stable random variables on different nodes (that is random variables $X_i$ and $X_j$) of a DAG are independent if and only if they have no common ancestors, see \cite[Theorem 2.3]{GKO}. 
As a consequence of this and percolation theory we find for the subcritical case $p \leq 1/2$ that two random variables become independent with probability 1, whenever their distance tends to infinity. In contrast, for the supercritical case there exists $\frac{1}{2} < p^* <1$ such that for $p >p^*$ two random variables are dependent with positive probability, even when their node distance tends to infinity.

Finally, we consider changes in the dependence properties of random variables on a sub-DAG $H$ of a finite or infinite graph on the oriented square lattice $\Z^2$, when enlarging this subgraph.
The method of enlargement consists of adding nodes and edges of  Bernoulli bond percolation clusters. 
Here we start with $X_i$ and $X_j$ independent in $H$, and answer the question, whether they can become dependent in the enlarged graph. 
We evaluate critical probabilities such that $X_i$ and $X_j$ become dependent in the enlarged graph with positive probability or with probability 1. 
We find in particular that for every DAG $H$ with finite number of nodes, in the enlarged graph $X_i$ and $X_j$ remain independent with positive probability.
On the other hand, if $H$ has nodes $\Z^2$ and percolates everywhere; i.e., every connected component of $H$ is infinite, then $X_i$ and $X_j$ become dependent with probability 1 in the enlarged graph.

The recursive max-linear process $X$ from Definition~\ref{MLP} below, may be viewed as a model for the communication between members of an infinitely large network, which may be regarded as an arbitrarily large union of individual networks of finite size, where each finite network has its own communication structure. These are represented by finite sub-DAGs. {A practical example in which a max-linear process is eligible as a model is given by the exploration of web-based communication or, more generally, complex networks in which it is of considerable interest to determine (the most) influential nodes. Concerning web-based communication as a (random) graph model nodes may be identified with ranks of a certain webpage, that is realizations of the random variables  $X_i$ may correspond to concrete values of ranks. When using the PageRank as a tool in order to detect influences, several results \cite{Jelen, Volko} approve that the distribution of a rank is heavy-tailed, giving rise to employ max-linear models as an alternative to capture the evolution of influences \cite{Marko}.
	
Besides, we believe the scope of applicability of the model under discussion is quite flexible. A concrete example, mentioned in more detail in Section 5, that we propose is the modeling the course of an auction. Numerous auction houses nowadays offer live auctions, in which bidders from all over the world can place their bids on the internet. We assess the max-linear models to be suitable in terms of modeling the course of such auctions. For further discussion on this topic we refer to Section 5.
	
Another practical application, particularly of statistical interest, is the identifiability of recursive max-linear processes from concrete observations and known DAGs. In particular, \cite{GKL} provides estimation procedures for crucial parameters in the model, namely the edge-weights and max-linear coefficient matrix discussed below. More precisely, for $n\in \mathbb{N}$ let $X^1, \ldots, X^n$ be independent realizations of a max-linear model, i.e., a random vector as given in Definition \ref{MLP} below. Assume that for each $X^j$, $1\leq j \leq n$, its distribution (which is assumed to have atom-free margins on $\mathbb{R}_+$) and the underlying DAG are known. Then, according to \cite[Section 4]{GKL} one can estimate the corresponding max-linear coefficient matrix without needing further conditions.

Our paper is organized as follows.
In Section~\ref{s2} we introduce recursive max-linear models on DAGs in $\Z^2$. 
In particular, we give sufficient conditions under which max-linear models on infinite graphs are well-defined. Section~3 uses the fact that the max-linear coefficients $b_{ji}$ originate from an algebraic path analysis by multiplying edge weights along a path between nodes $j$ and $i$ with $j$ being an ancestor of $i$. 
This concept, known from finite recursive max-linear models, extends to infinite DAGs.
Example~\ref{mw} shows that the important class of max-weighted models can be extended from finite to infinite graphs such that the max-weighted property remains. 
Recursive max-linear processes on a DAG have the nice property that independence between random variables on two different nodes is characterized by their ancestral sets. We prove that this also holds in the setting of infinite graphs.
This is the starting point of our investigation. 
Section~\ref{s4} contains the dependence results.
Here we investigate the Bernoulli bond percolation DAGs.  
In Section~\ref{s41} we prove that nearest neighbor bond percolation on $\Z^2$ yields independence of $X_i$ and $X_j$ with probability 1 provided $|i-j|\to\infty$ for $p < p^*$, whereas it yields dependence with positive probability for $p>p^*$ and some $\frac{1}{2} < p^* <1$.
In Section~\ref{s42} we investigate for $X_i$ and $X_j$, which are independent in some subgraph $H$, whether enlargement of $H$ can result in dependence between $X_i$ and $X_j$.
Finally, in Section \ref{s5} we discuss applications in communication networks in more detail and interpretations of our results in this context.

\section{Max-linear processes on directed acyclic lattice graphs}\label{s2}

This section presents a description of infinite max-linear models on directed acyclic lattice graphs. We first explain the structure of the directed graph on a lattice before we define and show the existence of a random field with finite-dimensional distributions entailing a dependence structure of max-linear type encoded in such graphs.

\subsection{Graph notation and terminology}\label{s21}

Let $\ganz^2$ be the oriented square lattice as follows (e.g. \cite{BBS,BR,D,Grimmett}). We write $i = (i_1, i_2)$ for elements in $\ganz^2$ and refer to them as nodes. The distance from $i$ to $j$ is defined in terms of the Manhattan metric given by
$$ \delta (i,j) =|i_1-j_1| + |i_2-j_2|$$
for $i,j \in \ganz^2$. We regard $\ganz^2$ as a graph by adding edges between all nodes $i,j$ with $\delta (i,j)=1$. In addition, we assume the edges to be oriented in the following manner. Denote by $\pa (i)$ and $\operatorname{ch} (i)$ the parents and children of node $i =(i_1, i_2)$, respectively. Then $j = (j_1, j_2) \in \pa (i)$ if and only if either $(j_1, j_2) = (i_1-1, i_2)$ or $(j_1, j_2) = (i_1, i_2-1)$ and, consequently, $j = (j_1, j_2) \in \operatorname{ch} (i)$ if and only if  either $(j_1, j_2) = (i_1+1, i_2)$ or $(j_1, j_2) = (i_1, i_2+1)$. 
We may write $i \rightarrow j$ if there is a directed edge from $i$ to $j$, that is if $i$ is a parent of $j$. 
The set of edges in this oriented lattice $\Z^2$ is $E (\ganz^2)$, which is a subset of  $\ganz^2 \times \ganz^2$. 
In this paper we work with graphs $G = \big( V(G), E(G) \big)$ with $V(G) \subset \ganz^2$ and $E(G) \subset E (\ganz^2)$, which are {{directed acyclic} lattice graphs}. 
{We refer to them simply as DALGs or DAGs.}
{When there is no ambiguity,} we often abbreviate $V =V(G)$ and $E = E(G)$. Thus, every node $i \in V$ has at most two children and two parents, but possibly infinitely many descendants and ancestors, denoted by $\operatorname{de} (i)$ and $\operatorname{an} (i)$, respectively. Moreover, we define $\De (i) = \{ i \} \cup \operatorname{de} (i)$ and $\An (i) = \{ i \} \cup \operatorname{an} (i)$. 
{Note} that such a DAG may have no roots, i.e., it might be the case that for all $i \in V$ we have $\an (i) \neq \emptyset$, which proves relevant for the questions we want to answer.

\subsection{Infinite recursive max-linear models}\label{s22}

We now introduce recursive max-linear processes. Let $G = ( V(G), E(G) )$ be a DAG with some possibly infinite set of nodes $V(G) \subset \ganz^2$. Moreover, we assume that all the nodes $i \in V(G)$ and  all the edges $(i,j) \in E(G)$ are equipped with prespecified (strictly) positive weigths $c_{ii}$ and $c_{ij}$, respectively. Recall from \cite[Section 1]{GK} that if $|V(G)| < \infty$ a recursive max-linear model $X = (X_i)_{i \in V(G)}$ on $G$ is given by
$$ X_i =\bigvee_{k\in\pa_G(i)} c_{ki} X_k \vee c_{ii} Z_i,\quad i\in V(G),$$
where $(Z_j)_{j\in V(G)}$ are independent continuously distributed non-negative noise variables with infinite support on $(0,\infty)$ and $\pa_G(i)$ denotes the parents of $i$ which belong to the DAG $G$. Recall that $c_{ii}$ is the weight of a node $i$. By \cite[Theorem 2.2]{GK}, applying a standard path analysis, the vector $X$ exhibits a max-linear structure, that is
\begin{equation} \label{crev1}
X_i =\bigvee_{j\in\an_G(i) \cup \{i\}} b^G_{ij} Z_j,\quad i\in V(G),
\end{equation}
where we denote by $\an_G(i)$ the ancestors of $i$ on the DAG $G$ and the matrix $B^G = (b^G_{ij})_{i,j\in V(G)}$ is given by
$$b^G_{ij} =\bigvee_{p\in P_{ji}(G)} d_{ji} (p) \textnormal{ for } j \in \operatorname{an}_G (i), \quad b^G_{ii} = c_{ii}, \textnormal{ and } b^G_{ij} =0 \textnormal{ for } j \in V(G) \setminus \big( \operatorname{an}_G (i) \cup \{i\} \big)$$
with $P_{ji}(G)$ denoting the set of all paths in $G$ from $j$ to $i$ and $d_{ji} (p)$ is defined by
$$d_{ji} (p) = c_{jj} \prod_{l=0}^{n-1} c_{k_l k_{l+1}}$$
for every path $p=[ j \rightarrow k_1 \rightarrow \ldots \rightarrow k_n = i]$. Note that this representation explicitly depends on $G$ and $B^G$ is called the max-linear coefficient matrix of $X$ with respect to $G$. We now provide an extension of this to infinite graphs, in which a family of infinitely many random variables is characterized by a graph $G$ with infinitely many nodes and edges. 

\begin{defi} \label{MLP}
	We call a family of random variables $X:=\{ X_i : i \in V(G) \}$ {recursive max-linear process} if for every $i \in \mathbb{Z}^2$ the random variable $X_i$ is given by the representation
	$$ X_i =\bigvee_{j\in\an_G(i) \cup \{i\}} b^G_{ij} Z_j,$$
	provided that the latter maximum is almost surely finite, $(Z_j)_{j\in V(G)}$ are independent continuously distributed non-negative noise variables with infinite support on $(0,\infty)$ and $b^G_{ji}$ is computed by the path analysis described above.
\end{defi} 

We now prove the existence of a stochastic process with the dependence structure described by infinite recursive max-linear processes as in Definition~\ref{MLP}. Furthermore, we give a sufficient condition on the weights under which there exists a stochastic process $X=\{ X_i : i \in V(G) \}$ as in Definition~\ref{MLP}.

We illustrate the procedure of extending max-linear models in case of two finite subgraphs of the lattice. Assume that $(V_1 = \{ i_1, \ldots i_m\} , E_1)$ and $(V_2 = \{ j_1, \ldots j_n\} , E_2)$ are two such finite subgraphs. Suppose that $X^1= (X_{i_1}, \ldots, X_{i_m})$ and $X^2= (X_{j_1}, \ldots, X_{j_n})$ are the corresponding recursive max-linear models with coefficient matrices $B_1$ and $B_2$, respectively with recursive max-linear representation
\begin{align*}
	X_{i_k} &= \bigvee_{r\in\pa^1(i_k)} c_{r i_k} X_r\vee c_{i_k i_k} Z_{i_k},\quad k \in \{1,\dots,m \},	&X_{j_l} = \bigvee_{s\in\pa^2(j_l)} c_{s j_l} X_s\vee c_{j_l j_l} Z_{j_l},\quad l\in \{1,\dots,n\},
\end{align*}
with $\pa^1$ and $\pa^2$ denoting the parents with respect to the graphs $(V_1, E_1)$ and $(V_2, E_2)$, respectively. Consider the enlarged finite graph $(V= V_1 \cup V_2, E=E_1 \cup E_2)$. 
Then a recursive max-linear model on this graph is given by
\begin{align*}
	Y_{i_k} &= \bigvee_{r\in \pa^{1,2}(i_k)} c_{r i_k} Y_r\vee c_{i_k i_k} Z_{i_k},\quad \quad k\in \{1,\dots,m\}, &Y_{j_l} = \bigvee_{s\in \pa^{1,2}(j_l)} c_{s j_l} Y_s\vee c_{j_l j_l} Z_{j_l},\quad l\in \{1,\dots,n\},
\end{align*}
with $\pa^{1,2}$ denoting the parents with respect to the graph $(V, E)$ and coefficient matrix $B^{1,2}$ calculated by the usual path analysis (see \cite[Theorem 2.2]{GK}). In the following for notational simplicity write $ b_{ij} =  b^G_{ij}$.

\begin{lemma} \label{MLPexist}
	Let $\alpha \in (0, \infty)$ and assume that
	\begin{align} \label{revcon}
		\sum_{j \in \mathbb{Z}^2} (b_{ij})^\alpha < \infty \quad \forall i \in \mathbb{Z}^2.
	\end{align}
	Moreover, assume that $(Z_{k})_{k \in \ganz ^2}$ is a sequence of independent standard $\alpha$-Fr\'echet distributed noise variables. Then there exists a max-linear process as in Definition~\ref{MLP}.
\end{lemma}

\begin{proof}[\textbf{\upshape Proof:}]
	We prove that the weighted maximum of infinitely many noise variables is finite with probability one. Indeed, let $x \in (0, \infty)$ and define
	$$X_i =\bigvee_{j\in\An_G(i) } b_{ij} Z_j$$
	for $ i \in \mathbb{Z}^2.$ Then
	\begin{align*}
	\operatorname{Pr}(X_i \leq x) & \leq \operatorname{Pr} \big( b_{ij} Z_j \leq x \, ,\forall j \in \An_G(i) \big) = \prod\limits_{j \in \An_G(i)} \operatorname{Pr}(Z_j \leq \frac{x}{b_{ij}} ) = \prod\limits_{j \in \An_G(i)} \exp \big( -x^{-\alpha}  (b_{ij})^\alpha \big) \\
		& = \exp \Big(  -x^{-\alpha} \sum_{j \in \An_G(i)} (b_{ij})^\alpha \Big) \in (0,1)
	\end{align*}
	by condition \eqref{revcon}. Thus, $X_i$ has a Fr\'echet distribution. Moreover, let $i_1, \ldots, i_d \in \mathbb{Z}^2$, $d \geq 1$, and $x_{i_1}, \ldots, x_{i_d} \in (0, \infty)$. Then, by a simple calculations, the finite-dimensional distributions of $X=\{ X_i : i \in V(G) \}$ are given by
	\begin{align*}
		\operatorname{Pr}(X_{i_1} \leq x_{i_1}, \ldots, X_{i_d} \leq x_{i_d}) &= \operatorname{Pr} \Bigg( \bigvee_{j\in\An_G(i_k) } b_{i_kj} Z_j \leq x_{i_k} , k=1, \ldots, d \Bigg) = \prod_{k=1}^{d}\operatorname{Pr} \Bigg(  Z_j \leq \bigwedge_{i_k \in\De_G(j) } \frac{x_{i_k}}{b_{i_kj}} \, , \forall j \in V(G) \Bigg) \\
		& = \exp \Big( -\sum_{j=1}^{d} \bigvee_{k\in\De_G(i_j) } \big( \frac{b_{k{i_j}}}{x_k} \big) ^\alpha \Big).
	\end{align*}
	
	In particular, every recursive max-linear process in which the noises are standard $\alpha$-Fr\'echet as in Definition~\ref{MLP} exhibits these finite-dimensional distributions.
\end{proof}

Consider the following example of a (max-weighted) max-linear process $X$ with weights satisfying the assumption \eqref{revcon}, see also Example \ref{mw} below. For simplicity assume that for every $i,j \in \mathbb{Z}^2$ with $j \in \An (i)$ there is only one path from $j$ to $i$. Let $p=[ j=k_0 \rightarrow k_1 \rightarrow \ldots \rightarrow k_n = i]$ be the path from $j$ to $i$. Assume that the edges are equipped with the weights
$$c_{k_\nu k_{\nu+1}} = \Big( \frac{1}{|k_\nu|+1} \Big)^{2}, \quad 0 \leq \nu \leq n-1,$$
where $c_{ii} =1$ for every $i$. Note that condition \eqref{revcon} is satisfied. In particular, since the weights are vanishing, this shows that the larger the distance between a node and its ancestor, the smaller the contribution of the ancestor, which is a natural property to hold.

{Different blocks of the matrix $B$ may correspond to distinct communities with different communication structure. The values of the random variables $X_i$ may correspond to extreme observations.}

The following limit result, which can be found in \cite[Lemma~2.1(iv)]{ST}, shows that we can regard a max-linear model on an infinite graph as a limit of a sequence of max-linear models on finite graphs. We precise this now in the following Remark. 

\begin{defi}
	In the following we say that a sequence of subgraphs $(V_n, E_n)_{n\in\N}$ of  a graph $(V,E)$ tends to $(V,E)$ if for every $j \in V$ and $e \in E$ there exists $n \in \mathbb{N}$ such that also $j \in V_m$ and $e \in E_m$ for every $m \geq n$.
\end{defi}

\brem
If $(Z_j)_{j\in\Z^2}$ are independent standard $\al$-Fr\'echet random variables and $(V_n,E_n)_{n\in\N}$ is a sequence of finite sub-DAGs of the oriented square lattice $\Z^2$ 
then from Lemma~\ref{MLPexist} we know that for each $n \in \mathbb{N}$
\beam\label{Xn}
X_i^{(n)} = \bigvee _{j \in V_n} b_{ij} Z_j, \quad i \in V_n,
\eeam
has $\al$-Fr\'echet distribution with scale parameter $(\sum_{j\in V_n} (b_{ij})^\al)^{1/\al}$.
Suppose that the sequence of DAGs $(V_n, E_n)_{n\in\N}$ tends to a DAG $(V,E)$ with infinitely many nodes as $\nto$ and that
$$ b:=\lim_{n \to \infty} \sum_{j\in V_n} (b_{ij})^\al$$
exists. Then $X_i^{(n)}\,\stas\, X_i,$ $\nto,$ where $X_i$ has $\al$-Fr\'echet distribution with scale parameter $b^{1/\al}<\infty$. If the series 
$$\lim_{n \to \infty} \sum_{j\in V_n} (b_{ij})^\al$$
diverges then $X_i^{(n)}\stas \infty$ as $\nto$.
\erem

{Provided $X_i$ is almost surely finite, the value at node $i$ may originate in a large number of values along an infinite path. 
	As there may be many sequences of subgraphs with limit $(V,E)$ the random variable at node $i$ depends on this sequence. There may be sequences of subgraphs or paths in subgraphs leading to very large values of $X_i$, as a consequence, all its descendants also become large.}


We now treat the case that $V(G) \subset \nat_0^2 = \nat_0\times \nat_0$, so that every node has at most finitely many ancestors and give an example of a {{max-weighted process}}. To this end,we consider infinite DAGs on $\N^2_0$, which we view as a prototypical sub-DAG with infinitely many nodes of the oriented square lattice $\Z^2$, such that each node has at most finitely many ancestors. 

\subsection{Max-weighted process}

Let $G = (V,E)$ be a DAG with $V \subset \nat_0^2 = \nat_0\times \nat_0$ 
and corresponding edges $E$. Assume a recursive max-linear process $X = \{ X_i : i \in V\}$ on $G$. In the following the aim is to give a canonical choice of a possible max-linear coefficient matrix $B$ associated with the process $X$ and to introduce a process that we call max-weighted.

Assume that the edges of $G$ are equipped with positive weights $c_{ki}$ for every $i \in V$ and {$k \in  \{i\}\cup\pa (i)$}. For $n \in \nat$ let $G_n = (V_n, E_n)$ be the DAG with nodes $V_n = \{ i = (i_1, i_2) \in V : i_1 + i_2 \leq n \}$ and corresponding edges taken from $E$, so that $\lim_{n \to \infty} G_n = G.$ 
By Definition~\ref{MLP} there are independent non-negative noise variables  {$(Z_i)_{ i \in V_n}$ with infinite support on $(0,\infty)$}  and a max-linear coefficient matrix $B = (b_{ij})_{i,j\in V_n}$ with non-negative entries such that
$X_i^{(n)}$ is as in \eqref{Xn}.
Indeed the entries $b_{ij}$ may be derived from the path analysis mentioned in Section 2. 
This in particular shows that  for $i \in V$ the $b_{ij}$ do not depend on the descendants $\operatorname{de} (i)$. 
{Thus, an infinite max-linear coefficient matrix $B$ is built up from increasing finite blocks representing $V_n$ for increasing $n\in\N$.}

{For a communication network on $\N_0^2$ the representation \eqref{Xn} reduces to a maximum over finitely many random variables, for instance, the root $0$ influences all the other nodes in the network. Hence, if the root node happens to hold the maximum of all $Z_j$ for $j\in\N_0^2$ its influence may dominate the whole network, although by the max-linear coefficient matrix $B$ all other nodes may have different realisations.}

As there may be several paths between nodes with different path-weights, so-called max-weighted models with same paths-weights along all possible directed paths between two nodes play an important role.
We now give an example of such a max-linear process relying on the definition of max-weighted models presented in \cite[Definition 3.1]{GK} and discussed in \cite[Section 3]{GKO}.  Resulting as a limit of max-weighted paths, we may call such a process max-weighted.

\begin{example} [Max-weighted process]\label{mw}
	Let $V = \nat_0^2$ be the set of nodes and assume oriented edges between all nodes $i,j$ with $\delta (i,j) = 1$. Start with a subgraph in which the set of nodes is bounded and of the form $V_n = \{ (i_1,i_2) \in \mathbb{N}_0^2 : i_1+i_2 \leq n \}$ for some $n \in \mathbb{N}_0$ and the corresponding set of edges is denoted by $E_n$. Assume that the corresponding model is max-weighted so that every entry of the max-linear coefficient matrix is given by $b_{ji} =  d_p \big( (j_1,j_2) , (i_1,i_2) \big)$, where $d_p \big( (j_1,j_2) , (i_1,i_2) \big)$ is calculated by a path analysis along the edge-weights as in equation (1.5) in \cite{GK}. Since the model is max-weighted, $d_p \big( (j_1,j_2) , (i_1,i_2) \big)$ is the same value for every path $p$ from $i$ to $j$ and thus  we can write $d_p \big( (j_1,j_2) , (i_1,i_2) \big) = d \big( (j_1,j_2) , (i_1,i_2) \big)$, since the latter value is independent of the chosen path $p$. We now show that the DAG can be enlarged in such a way that the enlarged new subgraph is again max-weighted. Moreover, this procedure can be executed infinitely often. 
	Let {$n \geq 1$} and assume that we add a node, say $(\ell_1,\ell_2)$ which we connect with the nodes $(\ell_1-1,\ell_2)$ and $(\ell_1,\ell_2-1)$ in $V$ by two edges with corresponding weights $c \big( (\ell_1-1,\ell_2), (\ell_1,\ell_2) \big)$ and $c \big( (\ell_1,\ell_2-1), (\ell_1,\ell_2) \big)$. By choosing these appropriately we can ensure that the new model is again max-weighted. More precisely, we choose the weights satisfying
	$$ c \big( (\ell_1-1,\ell_2), (\ell_1,\ell_2) \big) = \frac{c \big( (\ell_1,\ell_2-1), (\ell_1,\ell_2) \big) \cdot d \big( (0,0) , (\ell_1-1, \ell_2) \big)}{d \big( (0,0) , (\ell_1, \ell_2-1) \big)}.$$
	We now show that the enlarged DAG again leads to a max-weighted model. Let $p_1$ be a path from the root to $(\ell_1,\ell_2)$ containing $(\ell_1-1,\ell_2)$ and let $p_2$ be such a path containing the node $(\ell_1,\ell_2-1)$. Then we have by definition
	\begin{align*}
		d_{p_1} \big( (0,0) , (\ell_1,\ell_2) \big)  & = d \big( (0,0) , (\ell_1,\ell_2-1) \big) \cdot  c \big( (\ell_1-1,\ell_2), (\ell_1,\ell_2) \big) =  c \big( (\ell_1,\ell_2-1), (\ell_1,\ell_2) \big) \cdot d \big( (0,0) , (\ell_1-1,\ell_2) \big) \\
		& =  d_{p_2} \big( (0,0) , (\ell_1,\ell_2) \big) .
	\end{align*}
	Thus every path from the root to $(\ell_1,\ell_2)$ is max-weighted and this shows that the new model is max-weighted.
\end{example}

{In the following section we return to DAGs on $\Z^2$, which allow for infinitely many ancestors.} We consider percolation (dependence) properties between two fixed nodes $i$ and $j$ on $\Z^2$.

\section{Common ancestors and dependence structure}\label{s3}

In this section we let $G = (V,E)$ be an arbitrary, possibly infinite DAG with nodes
$V \subset \ganz^2$ and oriented edges $E$. Furthermore, we let $X$ be a recursive max-linear process on $G$ as in Definition~\ref{MLP}.

The following result is an analogue to \cite[Theorem 2.3]{GKO} and its proof {justifies the extension of the arguments to infinite dimension.}

\begin{prop} \label{TDC}
	Let $X:=\{ X_u : u \in V(G) \}$ be a recursive max-linear process and $i,j \in V(G)$. The following statements are equivalent: \newline \newline
\textnormal{(i) $X_i$ and $X_j$ are independent,} \newline
\textnormal{(ii) $\operatorname{An} (i) \cap \operatorname{An} (j) = \emptyset$}.

\end{prop}

\begin{proof}[\textbf{\upshape Proof:}]
	The proof extends \cite[Theorem 2.3]{GKO}. By Definition \ref{MLP} there exist independent noise variables $Z_k$, $k \in V(G)$, with infinite support on $(0, \infty)$ and a matrix $B = (b_{uk})$ such that
	$$ X_u =\bigvee_{k\in\An(u)} b_{uk} Z_k,\quad u\in V(G).$$
	Thus $X_i$ and $X_j$ are independent if and only if $\operatorname{An} (i) \cap \operatorname{An} (j) = \emptyset$. Indeed, first assume that $\operatorname{An} (i) \cap \operatorname{An} (j) = \emptyset$. Then we obtain
	\begin{align*}
		\operatorname{Pr} (X_i \leq x_i , X_j \leq x_j) & = \operatorname{Pr} \big( \bigvee_{k\in\An(i)} b_{ik} Z_k \leq x_i , \bigvee_{k\in\An(j)} b_{jk} Z_k \leq x_j \big) = \operatorname{Pr} \big( \bigvee_{k\in\An(i)} b_{ik} Z_k \leq x_i \big) \operatorname{Pr}\big(  \bigvee_{k\in\An(j)} b_{jk} Z_k \leq x_j \big) \\
		& = \operatorname{Pr} (X_i \leq x_i ) \operatorname{Pr}( X_j \leq x_j)
	\end{align*}
	for every $x_i, x_j \in (0, \infty)$, by independence of the noise variables $Z_k$, $k \in V(G)$. On the other hand assume that $X_i$ and $X_j$ are independent. By way of contradiction let us suppose that $\operatorname{An} (i) \cap \operatorname{An} (j)  \neq \emptyset$. Let $[-n,n]^2 = \{(x_1,x_2) \in \mathbb{Z}^2 : |x_1| + |x_2| \leq n \}$ for $n \in \N$. Consider a finite subgraph $H$ with $V(H) = [-n,n]^2 \cap V(G)$, where $n \in \mathbb{N}$ is so large such that
	$$\operatorname{An} (i) \cap \operatorname{An} (j) \cap [-n,n]^2  \neq \emptyset, \quad i,j \in [-n,n]^2,$$
	and $H$ contains all the edges of $G$ that are connecting nodes in $[-n,n]^2$. Write $V(H) = \{ i,j, i_1, \ldots, i_k\}$ for some $k \in \mathbb{N}$. Observe that $(X_i, X_j, X_{i_1}, \ldots X_{i_k})$ is a max-linear model on $H$ with almost surely finite, but not necessarily independent innovation noise variables given by
	$$\tilde{Z}_m =  \bigvee_{k\in \big(\An(m) \setminus V(H) \big) \cup \{m\}} b_{mk} Z_k , \quad m \in V(H).$$
	Let $l \in \operatorname{An} (i) \cap \operatorname{An} (j) \cap V(H)$. Then, by the assumptions on the noise variables $\tilde{Z}_k$, $k \in V(H)$, we have
	$$ \bigvee_{k\in\An_H(i)} b_{ik} \tilde{Z}_k = b_{il} \tilde{Z}_l, \quad \bigvee_{k\in\An_H(j)} b_{jk} \tilde{Z}_k = b_{jl} \tilde{Z}_l$$
	with positive probability, which implies that
	$$\operatorname{Pr} (X_i = \frac{b_{il}}{b_{jl}}X_j) >0.$$
	But by continuity of the noise variables, this contradicts the fact that $X_i$ and $X_j$ are independent. This finishes the proof.
\end{proof}

Having characterized the dependence between two random variables we are now interested in the following. We use Bernoulli bond percolation to generate random DAGs on the oriented square lattice $\Z^2$ and, thus, random dependence structures. 

{We want to answer the following question: given an extreme quantity,  observed at two nodes $i$ and $j$, is there a common cause in the network (a common ancestor) or not?}

\section{Bernoulli bond percolation DAGs}\label{s4}

The main purpose of this section is to {construct} max-linear models on randomly obtained DAGs with a possibly infinite number of nodes in order to investigate a randomized dependence structure. 

In view of Proposition~\ref{TDC} the probability that two random variables {$X_i$ and $X_j$ on the random graph} are {dependent} is nothing {else} than the probability that $i$ and $j$ have common ancestors inside the random open cluster containing nodes $i$ and $j$.
Our setting is a max-linear model on the oriented square lattice and percolation on this simple graphical model. 
This is a first step of linking percolation with max-linear models, and we envision further results on more sophisticated graphs as can be found, for instance, in \cite{Heydenreich} and \cite{Remco}.

\subsection{Max-linear models on random open clusters}\label{s41}

{Recall that we consider the oriented square lattice $\Z^2$. For this oriented model, the
	open cluster at 0 is usually defined as the set of all points we can reach from the origin by travelling along open edges in the direction of the orientation; see \cite{BBS,D}, or \cite[Section~12.8]{Grimmett}. As this open cluster always has root 0, all nodes $i$ and $j$ would have at least common ancestor 0, and would make the problem discussed below trivial. 
	Consequently, we consider unoriented, but not undirected, paths in \eqref{hamm2} as we will make precise below.
}

Let us first recall the framework of Bernoulli bond percolation from any  book on percolation as e.g. \cite{BR,Grimmett}.
Given the oriented square lattice $\Z^2$ with edge set $E\subset \Z^2\times \Z^2$, a {(bond) configuration} is a function $\omega : E\to \{0,1\}, e\mapsto\omega_e$.
An edge $e$ is open in the configuration $\omega$, if and only if $\omega_e=1$, so configurations correspond to open subgraphs. 
Recall from Section~\ref{s2} that in our setting open edges are directed, hence a configuration is
a DAG {denoted} by $(V,E)$ with $V \subset \mathbb{Z}^2$  and directed edges $E$. 
Each edge is declared open with probability $p$ and closed otherwise, different edges having independent designations. 
This gives the Bernoulli measure $P_p$, $p \in [0,1]$ on the space $\Omega = \{ 0,1 \} ^E$ of configurations. 
The $\sigma$-field $\mathcal{F}$ is generated by the finite-dimensional cylinders of $\Omega$. 
In summary, the probability space is $(\Omega, \mathcal{F}, P_p)$.

Let $C(k)$ be the {open cluster} containing the node $k\in V$.
The distribution of $|C(k)|$ is, by the translation-invariance of the measure $P_p$, well-known to be independent of $k \in V$, so that we assume in the following $k=0 \in V$ without loss of generality. 
If $|C(0)|$ denotes the (random) number of nodes of $C(0)$ then $P_p ( |C(0)| = \infty)$ is called the {percolation probability}. This probability depends on $p\in [0,1]$, and 
{Hammersley's critical percolation probability} is defined as
\begin{equation} \label{hamm}
p^1_c (V) = \inf \{ p \in (0,1) : P_p \big( |C(0)| = \infty \big) > 0 \}.
\end{equation}
Thus, for $p > p^1_c(V)$ it is possible to generate infinite open clusters with positive probability. 
{By Kolmogorov's zero-one law (cf. \cite[Theorem~1.11 ]{Grimmett}) there exists an infinite open cluster with probability 1 for $p> p_c^1(V)$, and otherwise with probability 0.}
Similarly, for two different given nodes $i,j \in V$ we can define $C(i,j)$ as the open cluster containing $i$ and $j$, with the convention that $C(i,j)= \emptyset$ if there is no path (of open edges) from $i$ to $j$. 
For notational simplicity in the following without loss of generality we assume that $j=0$.
{The following definition is related to the radius of a finite open cluster as investigated in \cite[Sections~6.1 and~8.4]{Grimmett}.}

As in \eqref{hamm} we define the critical probability 
\begin{equation}\label{hamm2}
p^2_c (V) = \inf \{ p \in (0,1) : P_p \big( |C(i,0)| = \infty \big) > 0 \},
\end{equation}
where we use the convention that $|C(i,0)| > 0$ if and only if there exists a (possibly undirected) path, 
\beam\label{openpath}
{[0  \leftrightarrow i]} := [ 0 = k_0 \leftrightarrow k_1 \leftrightarrow \dots \leftrightarrow k_n = i]
\eeam 
of open edges from $0$ to $i$, called an {open path}.

It is not difficult to see that $p^1_c (V) = p^2_c (V)$. Indeed, let $A = \{ 0 \leftrightarrow i \}$ be the event that there exists an open path from the origin to node $i$. 
Note that this event has strictly positive probability $P_p(0 \leftrightarrow i )$, also called the {two-point connectivity function} in \cite[Section~8.5]{Grimmett}. Thus
\begin{align*}
	P_p \big( |C(i,0)| = \infty \mid A \big)  = P_p \big( |C(0)| = \infty \mid A \big) .
\end{align*}
Let $\omega = (\omega _e)_{e\in E}, \omega ' = (\omega '_e)_{e\in E} \in \Omega$ with $\omega _e \leq \omega '_e$ for every $e \in E$. We recall that an event  {$A \subset \Omega$ is {increasing} if $\omega \in A$ implies that $\omega ' \in A$.} Since all the considered events are increasing, the Fortuin–Kasteleyn–Ginibre (FKG)-inequality \cite[Theorem 2.4]{Grimmett} further yields
\begin{align*}
	\frac{ P_p \big( |C(i,0)| = \infty \big)}{P_p(A)} = \frac{P_p \big( \{ |C(i,0)| = \infty\} \cap A \big)}{P_p(A)}  \geq P_p \big( |C(0)| = \infty \big).
\end{align*}
Since $P_p(A) >0$ altogether we obtain
$$ P_p \big( |C(i,0)| = \infty \big) > 0 \Leftrightarrow P_p \big( |C(0)| = \infty \big) > 0$$
and thus $p^1_c (V) = p^2_c (V)$. Recall that the critical percolation probability {$p^1_c (\Z^2)$} on the whole unoriented square lattice $\mathbb{Z}^2$ equals $\frac{1}{2}$ and moreover satisfies $P_{\frac{1}{2}}\big( |C(i,0)| = \infty \big) = 0$ (\cite[Chapter 11]{Grimmett}).

Given such {an infinite open cluster,} we are interested in the probability that the random variables $X_i$ and $X_j$ on the random DAG are independent. First, we give a formal definition of a max-linear model on a random environment.

\begin{defi}
	Let  $\{ X_u : u \in \ganz^2\}$ be a max-linear model. Let $\omega \in \Omega$ be a configuration, i.e. a realization of a sequence of iid Bernoulli random variables indexed by the possible edge-set, in which an edge $e \in E$ is present if and only if $\omega_e = 1$. Let $V(\omega)$ bet its corresponding set of nodes. The process $\{ X_u : u \in V(\omega) \}$ is called a {max-linear model in random environment}.
\end{defi}

From now on we suppose that $\{ X_u : u \in V(\omega) \}$ is a max-linear model in random environment and we investigate the probability $P_p \big( X_i \textnormal{ and } X_j \textnormal{ are independent} \big)$. That is to say, we are mainly interested in the max-linear process $\{ X_i : i \in C(i,0) \}$ on the random sub-DAG with nodes $V ( C(i,0) )$ and edges $E ( C(i,0) )$.

We observe that the events 
$$\{ X_i \textnormal{ and } X_j \textnormal{ are dependent} \} = \{ \operatorname{An} (i) \cap \operatorname{An} (0) \neq \emptyset \}$$
and $ \{ \operatorname{De} (i) \cap \operatorname{De} (0) \neq \emptyset \}$ are increasing as defined above.

Let 
\begin{align} \label{revsigma}
	\Sigma := \{ \operatorname{An} (i) \cap \operatorname{An} (0) \neq \emptyset \} \cup \{ \operatorname{De} (i) \cap \operatorname{De} (0) \neq \emptyset \}
\end{align}
denote the event that node $i$ and node $0$ have common ancestors or descendants. From arguments given below, it is not difficult to see that 
$$  \frac{1}{2} P_p ( \Sigma) \leq P_p \big( \{ \operatorname{An} (i) \cap \operatorname{An} (0) \neq \emptyset \} \big) .$$
The following lemma gives a refinement of this bound, which may be of interest in its own right.

\begin{lemma}\label{el}
	For $0 \leq p \leq 1$ we have
	$$ P_p \big( \{ \operatorname{An} (i) \cap \operatorname{An} (0) \neq \emptyset \} \big) \geq 1 - \big( 1 - P_p ( \Sigma) \big) ^{\frac{1}{2}}.$$
\end{lemma} 

\begin{proof}[\textbf{\upshape Proof:}]
	{As before, for notational simplicity assume that $j=0$. By translation invariance we find}
	\begin{align*}
		P_p \big( \{ \operatorname{De} (i) \cap \operatorname{De} (0) \neq \emptyset \} \big) 
		= P_p \big( \{ \operatorname{An} (i) \cap \operatorname{An} (0) \neq \emptyset \} \big) ,
	\end{align*}
	more precisely, $\{ \operatorname{De} (i) \cap \operatorname{De} (0) \neq \emptyset \}$ and $\{ \operatorname{An} (i) \cap \operatorname{An} (0) \neq \emptyset \}$ are two increasing sets of equal probability. Recall from inequality (11.14) in \cite[p. 289]{Grimmett} that if $A_1, \ldots, A_m$ are increasing events with equal probability then
	$$ P_p (A_1) \geq 1 - \Bigg( 1 - P_p\Big( \bigcup_{j=1}^m A_j \Big) \Bigg)^{\frac{1}{m}}.$$
	Using this inequality we get
	\begin{align*}
		P_p \big( \{ \operatorname{An} (i) \cap \operatorname{An} (0) \neq \emptyset \} \big) & \geq 1 - \big( 1 - P_p \big(  \{ \operatorname{An} (i) \cap \operatorname{An} (0) \neq \emptyset \} \cup \{ \operatorname{De} (i) \cap \operatorname{De} (0) \neq \emptyset \} \big) \big) ^{\frac{1}{2}} = 1 - \big( 1 - P_p( \Sigma) \big) ^{\frac{1}{2}} .
	\end{align*}
\end{proof}


In what follows we need the analog $C^\to(k) : = \operatorname{An} ( k ) \cup \operatorname{De} ( k )$ of the open cluster $C(k)$ containing $k\in V$ in the oriented square lattice.
We denote by
$P_p ( |C^\to(k)| = \infty )$
the probability that there exists an oriented path from $k \in \mathbb{Z}^2$ to $\infty$, which is by translation-invariance independent of $k$. 
In \cite[Section 3]{D} it is shown that
$$p^* := \inf \{p \in (0,1) : P_p \big( |C^\to(0)| = \infty \big) > 0\}$$
holds for some critical probability $\frac{1}{2} < p^* < 1$.
The exact value for $p^*$ is unknown; however, it is known that $0,6298\le p^*<0,6735$ (\cite[Chapter 10]{Grimmett} and \cite{BBS}).

\begin{theorem}\label{th4.1}
	There exists $\frac{1}{2} < p^*<1$ with the following properties. For $p < p^*$ we have
	\beam\label{asind}
	\lim_{|i-j| \to \infty} P_p ( X_i\textnormal{ and } X_j \textnormal{ are independent} ) = 1.
	\eeam
	For $p >p^*$ there exists a constant ${0<C < 1}$ not depending on $|i-j|$ with
	\beam\label{asdep}
	{0 <} P_p ( X_i\textnormal{ and } X_j \textnormal{ are independent} ) \leq C.
	\eeam
\end{theorem}

\begin{proof}[\textbf{\upshape Proof:}]
	By translation-invariance the distribution of the above event only depends on the edge distance $|i|$. 
	We will make use of results on oriented percolation as discussed in \cite{D}. 
	In particular, in \cite[Section 7]{D} it is shown that
	$$ P_p ( |C^\to(k)| \geq n ) \leq C e^{-\gamma n}$$
	for some $C>0, \gamma >0$ decays exponentially as $n \to \infty$ for $p < p^*$, where $p^*$ is introduced above. From this and from Proposition~\ref{TDC} for every $p< p^*$ we obtain
	\begin{align*}
		P_p \big( X_i\textnormal{ and } X_j \textnormal{ are dependent} \big) & = P_p \big( \{ \operatorname{An} (i) \cap \operatorname{An} (0) \neq \emptyset \} \big) 
		\leq P_p ( |C^\to(0)| \geq |i|  \big)  \to 0
	\end{align*}
	as $|i|= |i-j| \to \infty$, giving \eqref{asind}. \\[2mm]
	In order to prove the second statement  we assume that $p > p^*$. Furthermore, let $\Sigma$ be the event in \eqref{revsigma} and let $\Sigma^\complement$ be its complement, which is the event that $i$ and $j$ have neither common ancestors nor descendants. Applying Kolmogorov's zero-one law one can easily deduce that for $i,j \in \mathbb{Z}^2$
	$$ P_p \Big( \Sigma^\complement \mathlarger{\mid} |C^\to(0)| = \infty, |C^\to(i)| = \infty \Big) = 0,$$
	which implies that
	$$  P_p \big( |C^\to(0)| = \infty, |C^\to(i)| = \infty \big) =  P_p \Big( \big\{ |C^\to(0)| = \infty, |C^\to(i)| = \infty \big\} \cap \Sigma \Big) \leq P_p(\Sigma ) .$$
	Hence, by Lemma~\ref{el} we can estimate for $p \neq 1$
	\begin{align*}
		1 & >  P_p \big( \{ \operatorname{An} (i) \cap \operatorname{An} (0) \neq \emptyset \} \big) \geq 1 - \big( 1 - P_p (\Sigma ) \big) ^{\frac{1}{2}}  {\geq 1 - \big( 1 - P_p \big( |C^\to(0)| = \infty, |C^\to(i)| = \infty \big) \big) ^{\frac{1}{2}} }\\
		& \geq 1 - \big( 1 - P_p \big( |C^\to(0)| = \infty \big) ^2\big) ^{\frac{1}{2}} > 0
	\end{align*}
	for every $|i|$, where the second last inequality follows from the FKG-inequality (\cite[Theorem 2.4]{Grimmett}). 
	Thus, in the supercritical phase, with positive probability one can generate {dependence between} random variables $X_i$ and $X_j$, which proves \eqref{asdep}.
\end{proof}

{Theorem~\ref{th4.1} links the subcritical and supercritical case to probabilities for dependence and independence of $X_i$ and $X_j$. 
	
	{For the communication in a Bernoulli bond percolation network, we conclude that for edges being open (communication channels) with small probability, extreme observations at two different nodes become a.s. independent, when nodes are far apart. However, if edges are open with high probability then there is a positive probability that two extreme values are observed dependently; i.e., there may be a common source.}
	
	Also further properties of $X_i$ and $X_j$ within the oriented square lattice $\Z^2$ can be derived similarly using percolation properties. 
	The following remark gives an example.}

\begin{remark}[Number of common ancestors per pair of nodes:]
	Let $0\le p\le 1$ and $ A(i,j,n) := |\An (i) \cap \An (j) \cap B(n) |$ the number of common ancestors of $i$ and $j$ inside the box $B(n) = [-n ,n]^2$. 
	Then by an ergodic theorem (cf. \cite[Theorem 4.2]{Grimmett}) $P_p$-a.s. and in $L^1(P_p)$,
	$$ \frac{1}{|B(n)|} \sum_{k,\ell\in B(n)\atop |k-\ell|=|i-j|} |A(k,\ell,n)|^{-1} \to E_p(|\An (i) \cap \An (j) |^{-1}),\quad \nto.$$ 
\end{remark}

\subsection{Enlargement of DAGs using Bernoulli percolation}\label{s42}

Throughout this section fix two nodes $i,j \in \ganz^2$. We are again interested in dependence properties of the random variables $X_i$ and $X_j$. 
We write $\mathcal{P}$ for the property that $X_i$ and $X_j$ are {dependent,} and for a DAG $G$ we write $G \in \mathcal{P}$  if a max-linear model $X$ on $G$ has the property that the components $X_i$ and $X_j$ are {dependent. }

Suppose that $H = \big( V(H), E(H) \big)$, $V(H) \subset \ganz^2$, is a sub-DAG of the oriented square lattice $\Z^2$ containing $i,j$ such that $X_i$ and $X_j$ are independent on $H$, equivalently $\An (i) \cap \An (j) \cap V(H) = \emptyset$ by Proposition \ref{TDC}; i.e., $H \notin \mathcal{P}$. We utilize a method introduced in \cite{Okamura} in order to enlarge the sub-DAG $H$ by adding possibly infinitely many nodes and edges of open clusters and investigate the probability that $X_i$ and $X_j$ become dependent on the randomly enlarged DAG. 

{In the framework of communication in a network, if two extremes are observed seemingly independent, we investigate if a possible dependence could arise by a different network of a network member $i$, which are not present in the original network. The following results answer this question.}

Recall that for $k \in \ganz^2$ the open cluster containing $k$ is denoted by $C(k)$. 
The following definition is taken from \cite[Definition 1.1]{Okamura}. For an analogous definition of enlargement of percolating everywhere subgraphs as in Theorem \ref{R2} below we also refer to \cite{BT}.

\begin{defi} \label{Ok}
	{For $0\le p\le 1$} let $U(H) = U (\omega, p, H)$ be the random subgraph of the oriented square lattice $\Z^2$ with node set 
	$$ V\big( U(H) \big)  = \bigcup_{k \in V(H)} V(C(k))$$
	and edge set
	$$ E\big( U(H) \big) = E(H) \cup \bigcup _{k\in V(H)} E(C(k)) .$$
\end{defi}

Note that by definition $U(H)$ is a DAG containing the nodes $i$ and $j$, as $H$ is assumed to contain $i$ and $j$. 
Furthermore, we add finitely many or possibly infinitely many nodes, according as $p \leq \frac{1}{2}$ or $p > \frac{1}{2}$. Moreover, Definition \ref{Ok} corresponds to percolation with underlying probability measure $P_p^H$ on $\{ 0,1\}^{E(\mathbb{Z}^2)}$ satisfying 
\beam\label{PH}
P_p^H (\omega_e =1 ) = 1\,\,\mbox{if}\,\, e \in E(H)\quad\mbox{and}\quad P_p^H (\omega_e =1 ) = p \,\,\mbox{else},
\eeam
for all $\omega_e \in \{ 0,1\}^{E(\mathbb{Z}^2)}$. In addition, we have by definition that
\beam\label{meas}
P_p \big( U(H) \in \mathcal{P} \big) = P_p^H \big( \An (i) \cap \An (j) \neq \emptyset \big).
\eeam

{One prerequisite is the measurability of the event \eqref{meas},
	and we verify this by
	observing that $\{U(H) \in \mathcal{P} \}$ is equivalent to the existence of some $n \in \mathbb{N}$ such that $\operatorname{An} (i) \cap \operatorname{An} (j) \neq \emptyset$ holds on the ball $B(i,n) = \{ y \in \mathbb{Z}^2 : \delta (y,i) \leq n \}$ and, thus, $\{U(H) \in \mathcal{P}\}$ is determined by configurations of edges in a finite ball, and hence measurable.}

In analogy to \cite[Definition 1.3]{Okamura} we regard certain kinds of critical probabilities
\beam
\label{okcrit1} p_{c,1, \mathcal{P}, H} &:=& \inf \{ p \in [0,1] : P_p \big( U(H) \in \mathcal{P} \big) > 0 \}\\
\label{okcrit2} p_{c,2, \mathcal{P}, H} &:= &\inf \{ p \in [0,1] : P_p \big( U(H) \in \mathcal{P} \big) = 1 \}.
\eeam

We first remark that {$\{U(H) \in \mathcal{P}\}$ has positive probability for all $p> 0$,} such that $p_{c,1, \mathcal{P}, H} = 0$ holds, and the interesting question is for which choice of sub-DAGs $H$ we have $p_{c,1, \mathcal{P}, H} = p_{c,2, \mathcal{P}, H}$. As an easy example we might first consider the non-connected DAG $H$ with node set $V(H) = \{ i,j \}$ and $E(H) = \emptyset$. It is straightforward to see that $P_p ( U(H) \notin \mathcal{P}) > 0$ for every $p \in [0,1)$ and this implies $p_{c,2, \mathcal{P}, H} = 1 \neq p_{c,1, \mathcal{P}, H}$. 
On the other hand, the following Lemma gives an example of a DAG, where the latter assertion is not true, i.e. $p_{c,1, \mathcal{P}, H} = p_{c,2, \mathcal{P}, H} = 0$. 

\begin{lemma} \label{4ex} 
	Let $H$ be an infinite DAG with nodes $V(H) = \ganz^2$ and let $k\in \Z^2$ such that $i_1 \leq k_1 \leq j_1$. 
	Assume edges $E(H)$ only inside the set
	$$\Big( \mathbb{Z}^2 \setminus \{(k_1 \pm 1, i_2 - n) : n \in \mathbb{N}_0 \} \Big) \times \Big( \mathbb{Z}^2 \setminus \{(k _1\pm 1, i_2 - n) : n \in \mathbb{N}_0 \} \Big).$$
	Then $p_{c,2, \mathcal{P}, H} = 0$.
\end{lemma}

\begin{proof} [\textbf{\upshape Proof:}]
	Fix $p \in (0,1)$. We show that $p_{c,2, \mathcal{P}, H} \leq p$ by calculating $P_p ( U(H) \notin \mathcal{P} )$. By choice of $H$ the event $\{ U(H) \notin \mathcal{P} \}$ does not depend on finitely many edges, see also Figure~\ref{Fig1}. {Hence, by Kolmogorov's zero-one law,}
	$$ P_p \big( U(H) \notin \mathcal{P} \big) \in \{ 0,1 \}.$$
	From $p \in (0,1)$ we further get $P_p ( U(H) \notin \mathcal{P} ) < 1$ and therefore $P_p ( U(H) \notin \mathcal{P} ) = 0$. This yields $P_p ( U(H) \in \mathcal{P} ) =1$ for every $p \in (0,1)$ and {concludes} the proof.
\end{proof}

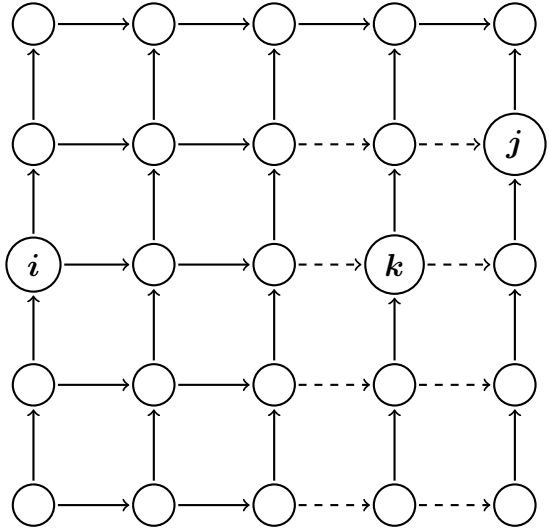
\begin{figure}
	\begin{center} 
		\begin{center}
			\hspace*{-1.1em}
			\begin{tikzpicture}[->,every node/.style={circle,draw},line width=0.8pt, node distance=1.6cm,minimum width=0.55cm,outer sep=0.5mm]
			\node (11) {${\boldsymbol {i}}$};
			\node (21) [right of=11] {};
			\node (31) [right of=21] {};
			\node (41) [right of=31] {${\boldsymbol k}$};
			\node (51) [right of=41] {};
			\node (10) [above of =11]{};
			\node (20) [above of=21] {};
			\node (30) [above of=31] {};
			\node (40) [above of=41] {};
			\node (50) [above of =51]{${\boldsymbol j}$};
			\node (12) [below of =11]{};
			\node (22) [below of =21]{};
			\node (32) [below of =31]{};
			\node (42) [below of =41]{};
			\node (52) [below of =51]{};
			\node (13) [below of =12]{};
			\node (23) [below of =22]{};
			\node (33) [below of =32]{};
			\node (43) [below of =42]{};
			\node (53) [below of =52]{};
			\node (14) [above of =10]{};
			\node (24) [above of =20]{};
			\node (34) [above of =30]{};
			\node (44) [above of =40]{};
			\node (54) [above of =50]{};
			\foreach \from/\to in {11/21,21/31,10/20,20/30, 12/22,22/32,13/23,23/33}
			\draw (\from) -- (\to);
			\foreach \from/\to in {11/10,21/20,31/30,41/40,51/50,12/11,22/21,32/31,42/41,52/51, 13/12,23/22,33/32,43/42,53/52}
			\draw (\from) -- (\to);
			\foreach \from/\to in {10/14,20/24,30/34,40/44,50/54,14/24,24/34,34/44,44/54}
			\draw (\from) -- (\to);
			\foreach \from/\to in { 31/41,41/51,30/40,40/50, 32/42,42/52,33/43,43/53}
			\draw[dashed] (\from) -- (\to);
			\end{tikzpicture}
		\end{center}
	\end{center}
	\caption{\label{Fig1} Visualization of one possible example in Lemma~\ref{4ex}. Lines indicate edges, which may be present or not; dashed lines indicate edges not allowed in $H$.}
\end{figure}

If we inspect the examples presented so far we recognize that the number of nodes and edges of the chosen DAG $H$ has a strong impact on whether we have $p_{c,1, \mathcal{P}, H} = p_{c,2, \mathcal{P}, H}$ or not. The following result substantiates this observation.

\begin{theorem}\label{R1}
	Let $H$ be a DAG and $j\in V(H)$ such that the cluster containing $j$ is finite. Then we have $p_{c,2, \mathcal{P}, H}=1.$
\end{theorem}

\begin{proof}[\textbf{\upshape Proof:}]
	Let $p <1$ and recall that
	$$P_p \big( U(H) \in \mathcal{P} \big) = P_p^H \big( \An (i) \cap \An (j) \neq \emptyset \big) .$$
	We prove the assertion by making use of planar duality arguments discussed in \cite[Section~1.4]{Grimmett}. Let $\mathbb{L}_d$ be the dual graph of $\ganz^2$ with nodes given by the set $\{ x + (\frac{1}{2}, \frac{1}{2}) : x \in \ganz^2 \}$ and edges joining two neighboring nodes so that each edge of $\mathbb{L}_d$ is crossed by a unique edge of its dual $\ganz^2$. As introduced in \cite[Section~1.4, p. 16]{Grimmett} an edge of the dual is declared to be open if it crosses an open edge of $\ganz^2$ and closed otherwise. Recall that a circuit of $\mathbb{L}_d$ is an alternating sequence $k_0, e_0, k_1, e_1, \ldots, k_n, e_n, k_0$ of nodes $k_0, \ldots, k_n$ and edges $e_0, \ldots, e_n$ forming a cyclic path from $k_0$ to $k_0$.
	
	Let $A$ be the event that there is a sub-path of closed edges of a circuit containing $j$ in its interior and $i$ in its exterior. Since the connected component containing node $j$ is finite, we have
	$$ 0 < P_p^H (A) \leq  P_p^H \big( \An (i) \cap \An (j) = \emptyset \big)$$
	which yields $ P_p \big( U(H) \in \mathcal{P} \big) <1$,
	for every $p \in [0,1)$. Thus, by definition we get $p_{c,2, \mathcal{P}, H}=1$ as claimed.
\end{proof}

\begin{cor} \label{bem}
	Let $H$ be a finite DAG. 
	Then we have $p_{c,2, \mathcal{P}, H}=1.$
\end{cor}

\begin{remark}
	Corollary \ref{bem} enlightens the fact that the events $\{ \An (i) \cap \An (j) \neq \emptyset \}$ and $\{ i \leftrightarrow j\}$ are essentially different. Indeed, if we choose a DAG $H \notin \mathcal{P}$ with $\{ \De (i) \cap \De (j) \neq \emptyset \}$ we have for every $0\leq p<1$,
	$$P_p^H ( i \leftrightarrow j ) = 1, \quad P_p^H \big( \An (i) \cap \An (j) \neq \emptyset \big) < 1,$$
	where $\{i \leftrightarrow j \}$ denotes the event that there is a path of open edges from $i$ to $j$. 
\end{remark}

Now we want to examine DAGs with the property that $p_{c,1, \mathcal{P}, H}= p_{c,2, \mathcal{P}, H} = 0$. In Lemma~\ref{4ex} we gave an example of a sub-DAG $H$ satisfying this equality.
We can prove the same identity for the class of percolating everywhere subgraphs, which is an analogous result to \cite[Theorem 1.13 (i)]{Okamura}. 
According to \cite{BH}, a {sub-DAG} $H$ is called {\em percolating everywhere} if $V(H) = \mathbb{Z}^2$ and every connected component of $H$ is infinite.

\begin{theorem} \label{R2}
	Let $H$ be a percolating everywhere sub-DAG of the oriented square lattice $\Z^2$. Then we have $p_{c,2, \mathcal{P},H}=0.$
\end{theorem}

\begin{proof}[\textbf{\upshape Proof:}]
	The proof partially relies on the proof of \cite[Theorem 1.13]{Okamura}. 
	As there we work with the probability measure $P_p^H$ on $\{ 0,1\}^{E(\mathbb{Z}^2)}$
	given in \eqref{PH}.
	Let $J$ be the graph with node set
	$$V(J) = \{ (k_1, k_2) \in \ganz^2 : k_1 \leq i_1, k_2 \leq i_2 \} \cup \{ (k_1, k_2) : k_1 \leq j_1, k_2 \leq j_2 \} .$$
	Note that if $J$ is connected then $\An (i) \cap \An (j) \neq \emptyset$. 
	Define the equivalence relation $k \sim \ell$ {on $\Z^2$} if and only if $P_p^H ( k \leftrightarrow \ell) =1$. 
	Denote by $[k]$ the equivalence class containing $k$ and ${\Z^2/_{\sim}=} Z' = Z' (\omega)$ the (Bernoulli) quotient graph with node set given by
	$$ V(Z') = \{ [k] : k \in \mathbb{Z}^2\} .$$
	If $| V(Z') | = 1$ then 
	$$ P_p^H \big( U(H) \textnormal{ is connected and } V(U(H) )= \mathbb{Z}^2 \big) = 1.$$
	Thus, with probability one there exists $k \in \operatorname{An} (i) \cap J$ with $k \leftrightarrow j$ so that $\operatorname{An} (i) \cap \operatorname{An} (j) \neq \emptyset$.\\[2mm]
	Now assume that $| V(Z') | \geq 2.$ 
	For sets $A,B \subset \mathbb{Z}^2$ let
	$$ E(A,B) = \{ (a,b) \in E (\mathbb{Z}^2) : a \in A, b \in B \}.$$
	By the same arguments as in the proof of \cite[Theorem 1.13]{Okamura} we can choose a partition $V( Z') = A \cup B$, $A \cap B = \emptyset$ with $|E(A,B)| = \infty$. 
	At this point observe that the number of connected components of $H$ is infinite, otherwise we would have $|E(A, B)| < \infty$ for every partition $V(Z') = A \cup B$. 
	Thus, by an application of Kolmogorov's  zero-one law we have 
	$$P_p^H \big( \An (i) \cap \An (j) = \emptyset \big) \in \{ 0,1 \},$$
	so that $P_p \big( U(H) \in \mathcal{P} \big) \in \{ 0,1 \}.$ This in particular implies that $p_{c,2, \mathcal{P},H}=p_{c,1, \mathcal{P},H}=0$ by definition and concludes the proof.
\end{proof}

\section{Communication networks }\label{s5}

{As indicated before, the question we answer here by means  of a simple probabilistic model is the following: given an extreme observation in a communication network, observed at two nodes, is there a common cause (a common ancestor) in the network or in an enlarged network or not.}


	In terms of the propagation of influence, every node may be interpreted as a network member, a directed edge between two nodes may be seen as a communication channel, and the weights represent the degree of influence between two members. 
	A phase transition in such a network indicates the non-existence or existence of a common cause of extreme observations of two different network members.}

Probabilistic communication models using tools from percolation theory to investigate phase transitions in graph structures are numerous in the literature; see e.g., \cite[Ch.~13]{Grimmett}, \cite{Liggett}, and \cite[Part IV]{Newman}, 
to name only a few. They model spread of diseases, voter behaviour, optimal behaviour of market agents, etc. within nearest neighbor lattice graphs, in preferential attachment models, or in small-world networks.

A basic model is explained in \cite{Dreyer} as follows: the authors assume the network nodes to take values randomly in $\{ 0,1\}$ representing two possible states.
A network member changes its state provided enough neighbours share a different state.
In contrast to this simple model, in the present paper the community members at every node exhibit observations, which can be modeled by any distribution, thus allowing for a more refined analysis and larger scope of interpretation for applications. For example, as already mentioned in the introduction, we can model the course of an auction. In this sense a community member represents a bidder in an auction, and we observe the bid placed by this person. Hence, the bid (for example money in dollars) is modeled by the random variable $X$. Since the purpose of a bid is to overbid the previous offers, a propagation by means of max-linear behavior is plausible, in which the noise variables $Z$ represent the amount of money the bidder is willing to spend independently, and in several cases depending of the type of auction a heavy-tailed distribution might be required.
One possible and eligible question of interest is to understand cause and effect of such extreme observations.

\bexam\label{comm}
Consider two arbitrary choices of finite communication networks modeled by $X$ as in Definition \ref{MLP}. 
More precisely, let $H_1$ be the DAG with nodes represented by $V =\{1,2,3\}$ and edge-set $E=\{ (2,3)\}$ consisting of one single edge, i.e. we have three network members and only $X_2$ and $X_3$ communicate, where $X_3$ is influenced by $X_2$. We assume the second DAG $H_2$ to be obtained from $H_1$ simply by adding the edge $(1,3)$, i.e. $X_1$ and $X_3$ start to communicate and $X_3$ is influenced by more than one source. Assume that the nodes and edges are equipped with positive weights $c_{ij}$, $i,j \in \{1,2,3\}$, and for $i \neq j$ we have $c_{ij} \neq 0$ if and only if there is an edge from $i$ to $j$. We now want to characterize the communication activities with the aid of max-linear coefficient matrices. For two matrices $M_1, M_2$ of same size we write $M_1 \preceq^0 M_2$ if all non-zero entries of $M_1$ are also non-zero entries of $M_2$ and there exists a zero entry of $M_1$ which is a non-zero entry of $M_2$. Let $B_1$ and $B_2$ be the max-linear coefficient matrices corresponding to $H_1$ and $H_2$, respectively. Applying the path analysis mentioned in Section 2 (cf. Theorem~2.4 of \cite{GK}) we obtain
$$B_1=\begin{pmatrix}
c_{11} & 0&0 \\
0 & c_{22} & c_{22} c_{23}  \\
0 & 0 &  c_{33}
\end{pmatrix} , \qquad
B_2=\begin{pmatrix}
c_{11} & 0&c_{11} c_{13} \\
0 & c_{22} & c_{22} c_{23}  \\
0 & 0 &  c_{33}
\end{pmatrix},$$
so that $B_1 \preceq^0 B_2$. Note that this stems from the fact that $H_2$ contains the edge $(1,3)$ not included in $H_1$. Thus, inspecting zero entries of the max-linear coefficient matrix helps in detecting communication channels. 
\eexam

Such observation holds in general and we summarize it in the following result.

\begin{prop}
	Let $X$ be a max-linear process with node-set $V$ and let $H_1$ and $H_2$ be two DAGs over the same finite set of nodes $V^H \subset V$ and max-linear coefficient matrices $B_1$ and $B_2$, respectively. If $B_1 \preceq^0 B_2$ then $H_2$ has more communication channels than $H_1$.
\end{prop}

Theorem~\ref{th4.1} gives rise to the following obvious interpretation. 
For a network with only moderately many communication channels, extreme observations at two nodes, which are far apart, are a.s. independent. 
However, in a highly communicative network, there may be a common source for an extreme observation presented at a specific node.

We now want to interpret the results in Section~\ref{s42} concerning random DAGs obtained from Bernoulli bond percolation clusters. Randomly added nodes and edges correspond to the formation of additional communication channels.
Consider the probability $p$ of an edge being open in the original network. 
For high values of $p$ the influences are more likely to spread.
We investigate this in more detail for a DAG $H$.
Assume that members of $H$ hold additional communication channels outside the communication network.
We call the combined network a {network with randomly spreading influences}.
What is the probability that two network members with independent observations become influenced by the same source in the combined larger network? 

Theorems~\ref{R1} and Corollary~\ref{bem} describe a situation, where 
the answer rather depends on the number of participants in the network and not so much on the structure of communication channels. 
This observation may be helpful in order to detect extreme observations simply by considering how many agents are affected by the spread of influences. 
In a wide sense, our results propose that extreme influences are less likely to spread if less agents are affected, being more decisive than the structure of communication channels. 

\begin{example} [Continuation of Example~\ref{comm}] \label{comm2} To precise these arguments we again compare two finite networks $H_1$ and $H_2$. 
	By Corollary~\ref{bem} two independent observations become influenced with certainty by a common source inside a network with randomly spreading influences, if these influences disseminate almost surely and only in this case, regardless of the setup of connections inside the network. 
	Recall that here $p$ can be regarded as the probability that a communication channel emerges. 
	In such a case we have $p=1$, which may correspond to very strong influences.
\end{example}

Theorem~\ref{R2} on the other hand, describes the situation, where the network has already many communication channels itself. Only some links between large communication communities are missing. Then links between these large communication communities are created a.s. whenever some randomly spreading influence arrives in the network at all. 

\section*{Acknowledgments}

We take pleasure in thanking the Mathematical Research Institute Oberwolfach for an invitation to the Research in Pairs Programme. This visit was extremely stimulating and fruitful for our research. Moreover, we would like to thank Markus Heydenreich for careful reading and comments which helped us to improve the manuscript. Additionally, we thank the Editor, Associate Editor and referees for  helpful comments.

\bibliographystyle{myjmva}

\end{document}